\documentclass[11pt]{article}

\usepackage[english]{babel}
\usepackage{amssymb}
\usepackage{amsfonts,amsmath,mathtools}
\usepackage{graphics}
\usepackage{lipsum}
\usepackage[ruled]{algorithm2e}
\usepackage{float}
\usepackage{url}
\usepackage{subfig}
\usepackage{xcolor}
\usepackage[export]{adjustbox}
\usepackage{tikz}

\usepackage{bm}
\usepackage{enumitem}
\usepackage{faktor}

\newcommand\blfootnote[1]{%
  \begingroup
  \renewcommand\thefootnote{}\footnote{#1}%
  \addtocounter{footnote}{-1}%
  \endgroup
}

\newtheorem{Lem}{Lemma}[section]
\newtheorem{Prop}[Lem]{Proposition}
\newtheorem{Cor}[Lem]{Corollary}
\newtheorem{Thm}[Lem]{Theorem}

\newtheorem{Def}[Lem]{Definition}
\newtheorem{Rem}[Lem]{Remark}

\newtheorem{Expl}[Lem]{Example}

\newenvironment{proof}[1][Proof]{\textrm{\bf #1.} }{\hfill$\Box$\medskip\medskip}

\newcommand\Gin{\operatorname{gin}}

\newcommand\Shad{{\operatorname{Shad}}}

\newcommand\supp{{\operatorname{supp}}}

\newcommand\lex{{\operatorname{lex}}}
\newcommand\slex{{\operatorname{slex}}}

\newcommand\pd{{\operatorname{pd}}}
\newcommand\reg{{\operatorname{reg}}}

\newcommand\im{{\operatorname{im}}}
\newcommand\lcm{{\operatorname{lcm}}}

\newcommand\depth{{\operatorname{depth}}}
\newcommand\bcos{{\operatorname{b-cosize}}}
\newcommand\cosi{{\operatorname{cosize}}}

\newcommand\A{\mathcal{A}}

\def\NZQ{\mathbb}

\def\FF{{\NZQ F}}
\def\QQ{{\NZQ Q}}

\let\emptyset\varnothing

\begin{document}

\title{Projective dimension and Castelnuovo--Mumford regularity of $t$--spread ideals}	
\author{Luca Amata, Marilena Crupi, Antonino Ficarra}	

\newcommand{\Addresses}{{
\footnotesize
\textsc{Department of Mathematics and Computer Sciences, Physics and Earth Sciences, University of Messina, Viale Ferdinando Stagno d'Alcontres 31, 98166 Messina, Italy}
\begin{center}
 \textit{E-mail addresses}: \texttt{lamata@unime.it};  \texttt{mcrupi@unime.it}; \texttt{antficarra@unime.it}
\end{center}
}}
\date{}
\maketitle
\Addresses

\begin{abstract}
We study some algebraic invariants of $t$--spread ideals, $t\ge 1$, such as the projective dimension and the Castelnuovo--Mumford regularity, by means of well--known graded resolutions. We state upper bounds for these invariants and, furthermore, we identify a special class of $t$--spread ideals for which such bounds are optimal.
\blfootnote{
	\hspace{-0,3cm} \emph{Keywords:} monomial ideals, minimal graded resolution, Taylor resolution, Koszul resolution, $t$--spread ideals, regularity, edge ideals.
	
}
\end{abstract}
	
\section{Introduction} Let $K$ be a field and let $S=K[x_1,\dots,x_n]$ be the standard polynomial ring in $n$ variables $x_1,\dots,x_n$ with coefficients in $K$. The study of the algebraic invariants of graded ideals in $S$ is a central topic in commutative algebra. Among the graded ideals, monomial ideals play an essential role. Indeed,
for any graded ideal $I$ of $S$, if $K$ is a field of characteristic $0$, the generic initial ideal of $I$, $\Gin(I)$, which is a monomial ideal, preserves the extremal Betti numbers of $I$, and consequently,  it preserves the projective dimension and the Castelnuovo--Mumford regularity \cite{BCP}. On the other hand, a particular subclass of monomial ideals, the squarefree monomial ideals, are related to many combinatorial structures such as partial ordered sets, simplicial complexes, graphs and hypergraphs. Thus, to study the extremal combinatorics of such objects, one can use tools from commutative and homological algebra and examine the defining squarefree ideals associated to these objects. 

Recently, Ene, Herzog and Qureshi, \cite{EHQ}, have generalized the notion of squarefree ideals. Given an integer $t\ge 1$, a monomial $u=x_{i_1}x_{i_2}\cdots x_{i_d}$, $1\le i_1<i_2<\cdots<i_d\le n$ is said \emph{$t$--spread} if $i_{j+1}-i_j\ge t$, for all $j=1,\dots,d-1$. A $t$--spread ideal is a monomial ideal of $S$ generated by $t$--spread monomials. If $t=1$, a $t$--spread ideal is just a squarefree ideal. This latter class of ideals has been vastly studied by many authors. One expects that many of the results which hold for squarefree ideals have ``$t$--spread" analogues. In this direction, C. Andrei--Ciobanu has generalized the classical Kruskal--Katona theorem, \cite{CAC}. Many other authors, including the ones of this paper, have examined $t$--spread ideals, and especially the class of $t$--spread strongly stable ideals, discussing some algebraic invariants \cite{ATSpread21, AC2,AFC1,AFC2,AEL,RD}.

In this paper we focus our attention on the computation of projective dimension and regularity of $t$--spread ideals, for $t\ge 1$. Section \ref{sec:1} contains basic notations and
terminology from commutative algebra and combinatorics that will be used in the paper.
In Section \ref{sec:2}, if $I$ is a squarefree ideal of $S$, we rewrite both a result on an upper bound for the projective dimension of $I$ \cite[Corollary 8]{CLT} and a result on an upper bound for the regularity of $I$ given in \cite[Corollary 2.4]{HPV},  \emph{via} homological tools and handling of monomials (Theorem \ref{thm:mainteorragtspread}). Our techniques are completely different from those in \cite{CLT, HPV} and we believe that they can be further improved in order to achieve \emph{sharper} bounds for both the projective dimension and the regularity of a $t$--spread ideal, $t\ge 1$. 
Hence, we obtain a bound for the regularity of a $t$--spread ideal, for all $t\ge1$, generalizing the classical bound due to Hochster \cite{HOC} for the squarefree case ($t=1$). It is $n-(t-1)$ (Theorem \ref{thm:maximumregtspread}).
An example of a $t$--spread ideal whose regularity is given by such a value is the \textit{Pascal ideal of type $(n,t)$}, where $n$ is the number of the indeterminates of $S$ (Definition \ref{def:Pascal}). 
Finally, in Section \ref{sec:appl}, we provide some applications of the main results for edge ideals  and some related classes of ideals.\medskip

\section{Preliminaries}\label{sec:1}

Let $S=K[x_1,\dots,x_n]$ be a polynomial ring in $n$ variables over a field $K$ with the standard grading, \emph{i.e.}, each $\deg x_i =1$.  For any graded ideal $I$ of $S$, there exists the unique minimal graded free $S$--resolution\smallskip
\begin{equation*}\label{eq:free}
\FF: 0\rightarrow F_s\xrightarrow{\ d_s\ } F_{s-1}\xrightarrow{d_{s-1}}\cdots\xrightarrow{\ d_2\ }F_1\xrightarrow{\ d_1\ }F_0\xrightarrow{\ d_0\ }S/I\rightarrow0,
\end{equation*}
with $F_i=\bigoplus_jS(-j)^{\beta_{i,j}}$. The numbers $\beta_{i,j}= \beta_{i,j}(S/I)$ are called the graded Betti numbers of $S/I$, while $\beta_i(S/I) = \sum_j \beta_{i,j}(S/I)$ are called the total Betti numbers of $S/I$. Recall, that the \emph{projective dimension} and the \emph{Castelnuovo--Mumford regularity} of $S/I$ are defined, respectively, as follows:
\begin{align*}
\pd(S/I)& = \max\{i:\beta_i(S/I)\neq 0\},\\
\reg(S/I)&= \max\big\{\mbox{$j-i$ : $\beta_{i,j}(S/I)\ne0$, for some  $i$ and $j$}\big\} \\
&= \max\big\{\mbox{$j$ : $\beta_{i,i+j}(S/I)\neq 0$, for some $i$}\big\}.
\end{align*}
Note that $\pd(I) = \pd(S/I)-1$ and $\reg(I) = \reg(S/I)+1$.
More precisely, the projective dimension $\pd(I)$ is the length of a minimal graded free resolution of the graded ideal $I$. 
Furthermore, the regularity of $I$ is at least the smallest total degree of a
generator of $I$, and if all the minimal generators of $I$ lie in the same degree, then $I$ has linear free resolution precisely when that degree equals $\reg(I)$ \cite[Lemma 5.55]{MS}.\\

The invariants above have been generalized in \cite{BCP} by the notion of \emph{extremal Betti number}. 
\begin{Def}\label{def:extr1}
An $i$--th Betti number $\beta_{i,j}(I)\neq 0$ is extremal if $\beta_{p,q}(I) = 0$ for all $p$ and $q$ satisfying the
following three conditions: (i) $p \ge i$, (ii) $p-q \ge i-j$, and (iii) $q \ge j +1$.
\end{Def}

Now, let $I$ be a graded ideal of the polynomial ring $S$. 
$I$ is a (monomial) squarefree ideal of $S$ if it is generated by squarefree monomials of $S$.	\\

For a positive integer $n$, we denote by $[n]$ the set $\{1,\dots,n\}$.

Recall that a \textit{simplicial complex} $\Delta$ on the vertex set $[n]$ is a family of subsets of $[n]$, such that 
\begin{enumerate}
\item[--] $\{i\}\in\Delta$ for all $i\in[n]$, and 
\item[--] if $F\subseteq \Delta$, $G\subseteq F$, we have $G\in\Delta$. 
\end{enumerate}

The dimension of $\Delta$ is the number $d=\max\{|F|-1:F\in\Delta\}$. It is well known that for any squarefree ideal $I$ of $S$ there exists a unique simplicial complex $\Delta$ on $[n]$ such that $I=I_\Delta$, where $I_\Delta$ is the Stanley--Reisner ideal of $\Delta$, \emph{i.e.}, the ideal of $S$ generated by all squarefree monomials $x_{i_1}x_{i_2}\cdots x_{i_r}$, $1\le i_1 < i_2 < \cdots < i_r\le n$, with $\{i_1, i_2,\ldots, i_r\}\notin \Delta$ \cite{JT}. The Stanley--Reisner ideal, $I_\Delta$, above defined is a squarefree ideal of $S$.

The notion of squarefree ideal has been generalized in \cite{EHQ} by the concept of $t$--spread ideal.

Given an integer $t\ge0$, we say that a monomial $x_{i_1}x_{i_2}\cdots x_{i_d}$ of $S$ with $1\le i_1\le i_2\le\dots\le i_d\le n$ is $t$--spread, if $i_{j+1}-i_j\ge t$, for all $j=1,\dots,d-1$. If $t\ge1$, every $t$--spread monomial is a squarefree monomial. We say that a monomial ideal $I$ is $t$--spread if $I$ is generated by $t$--spread monomials. If $t\ge1$, every $t$--spread ideal is a squarefree ideal.

Using the same notations as in \cite{CAC}, we denote by $[I_j]_t$ the set of all $t$--spread monomials of degree $j$ of a monomial ideal $I$ of $S$ and by $M_{n,d,t}$ the set of all $t$--spread monomials of degree $d$ of the ring $S$. From \cite[Theorem 2.3, (d)]{EHQ} we have that
$$
|M_{n,d,t}|\ =\ \binom{n-(d-1)(t-1)}{d}.
$$

If $T$ is a non empty subset of $M_{n,d,t}$, we define the \textit{$t$--shadow} of $T$ to be the set
\[
\Shad_t(T)\ =\ \big\{x_iw\,:\, w\in T\, \mbox{and $x_iw$ is $t$--spread monomial, $i=1,\dots,n$}\big\}.
\]

Let $t\ge1$. To an arbitrary monomial ideal $I$ of $S$, one can associate the following vector of integers:
$$f_t(I)=(f_{t,-1}(I),f_{t,0}(I),\dots),$$
with
\[
f_{t,j-1}(I)=\ |M_{n,j,t}|-|[I_j]_t|,\,\, j\ge 0.
\]
$f_t(I)$ is called the $f_t$--vector of $I$.

We close the section with the following definitions from \cite{CAC, EHQ}.

For a monomial $u\in S$, the set $\supp(u)=\{i:x_i\ \text{divides}\ u\}$ is called the \emph{support} of $u$.

\begin{Def} A non empty subset $L$ of $M_{n,d,t}$ is called a $t$--spread strongly stable set, if the following property holds:
for all $u\in L$, all $j\in \supp(u)$ and all $1\le i< j$ such that $x_i(u/x_j)\in M_{n,d,t}$, we have $x_i(u/x_j)\in L$.

A $t$--spread monomial ideal $I$ is $t$--spread strongly stable if $[I_j]_t$ is a $t$--spread strongly stable set for all $j$.
\end{Def}

Let $\ge_{\slex}$ be the usual squarefree lex order \cite{AHH2}.

\begin{Def} A non empty subset $L$ of $M_{n,d,t}$ is said a \emph{$t$--spread lexsegment set} if
for all $u\in L$ and all $v\ge_{\slex}u$, $v\in M_{n,d,t}$, one has $v\in L$.

A $t$--spread ideal $I$ is called a $t$--spread lexsegment ideal if $[I_d]_t$ is a $t$--spread lexsegment set for all $d$.
\end{Def}

If $I$ is $t$--spread strongly stable, one can associate to $I$ a unique $t$--spread lexsegment ideal, denoted by $I^{t,\text{lex}}$, such that  $f_t(I)=f_t(I^{t,\text{lex}})$ \cite[Theorem 2.1]{CAC}. In general,  a $t$--spread  ideal $I$ may not have a $t$--spread lexsegment ideal with the same $f_t$--vector \cite[Remark 2.5]{CAC}. Indeed, it is still open the problem of classifying all $t$--spread ideals which have an associated $t$--spread lexsegment ideal with the same $f_t$--vector.


Set
\[\max(u) = \max\big\{i:i\in\supp(u)\big\},\,\, \mbox{for $u\ne 1$},\]
and $\max(u)=0$, if $u=1$.  \\

For a $t$--spread strongly stable ideal $I$, one has \cite{EHQ}:
\[\pd(I) = \max\{\max(u)-t(\deg(u)-1)-1: u \in G(I)\},\]
\[\reg(I) = \max\{\deg(u): u \in G(I)\},\]
where $G(I)$ is the unique minimal set of monomial generators of $I$.

\section{Projective dimension and regularity}\label{sec:2}

In this section, if $S = K[x_1, \ldots, x_n]$, we discuss some optimal bounds for the projective dimension and the regularity of squarefree ideals of $S$. It is well--known by a result of Hochster \cite{HOC} that for squarefree ideals the maximum value of the Castelnuovo--Mumford regularity is given by $n$.  
Indeed, $x_1\cdots x_n$ is the squarefree monomial of the largest degree in a polynomial ring in $n$ variables. 
For $t$--spread ideals with $t\ge 2$ this bound can be refined, as we will see in the sequel. \\

Let us start by quoting a result from \cite{EHGB}.

\begin{Thm}\label{thm:betaijlebij}
\textup{\cite[Theorem 4.26]{EHGB}} Let $M$ be a finitely generated graded $S$--module, 
\[\FF: \cdots\rightarrow F_1\rightarrow F_0\rightarrow M\rightarrow0,\]
a minimal graded free resolution of $M$ with $F_i=\bigoplus_jS(-j)^{\beta_{i,j}}$, and $$\mathbb{G}:\cdots\rightarrow G_1\rightarrow G_0\rightarrow M\rightarrow0$$ a graded free $S$--resolution of $M$ with $G_i=\bigoplus_jS(-j)^{\textup{b}_{i,j}}$. Then
$$
\beta_{i,j}(M)\ \le\ \textup{b}_{i,j}
$$
for all $i$ and $j$.
\end{Thm}

Let $I$ be a squarefree ideal of $S$ with minimal set of monomial generators given by  $G(I)=\{u_1,\dots,u_p\}$.

Set
\[\begin{aligned}
\A_0 &=  \big\{\supp(u): u\in G(I)\big\} = \big\{\supp(u_1),\dots,\supp(u_p)\big\},&\\
\A_i &= \Big\{\bigcup_{\ell=1}^{i+1}\supp(u_{j_\ell}): \{j_1 < j_2 < \cdots < j_{i+1}\}\subseteq [p]\Big\},\, \mbox{for $i\ge 1$}.&
\end{aligned}
\]
The integer
$$s = \min\Big\{i : \mbox{ for all $A \in \A_i$}, A = \bigcup_{\ell=1}^p\supp(u_\ell)\Big\}$$
will be called the \emph{support index} of $I$.
 
One can quickly observe that if the sets in $\A_0$ are pairwise disjoint, then the support index of $I$ is equal to $p-1=\vert G(I) \vert -1$.
\begin{Expl}\label{ex:sets} \rm Let $S=\QQ[x_1, \ldots, x_{11}]$ and consider the following squarefree ideal $I=(x_2x_4, x_1x_5x_7, x_3x_7x_9x_{11})$ of $S$. Setting $u_1 = x_2x_4$, $u_2 =x_1x_5x_7$ and $u_3 =  x_3x_7x_9x_{11}$, we have
\begin{eqnarray*}
\A_0&=&\big\{\supp(u_1),\supp(u_2),\supp(u_3)\big\} = \big\{\{2, 4\}, \{1, 5, 7\}, \{3, 7, 9, 11\}\big\},\\
\A_1&=&\Big\{\bigcup_{\ell=1}^{2}\supp(u_{j_\ell}): \{j_1 < j_2\}\subseteq [3]\Big\}\\
&=&\big\{\{1,2,4,5,7\}, \{2,3,4,7,9,11\},\{1,3,5,7,9, 11\}\big\},\\
\A_2&=&\Big\{\bigcup_{\ell=1}^{3}\supp(u_{j_\ell}): \{j_1 < j_2 < j_3\}\subseteq [3]\Big\}= \big\{\{1,2,3,4,5,7,9,11\}\big\}\\
&= &\Big\{\bigcup_{\ell=1}^3\supp(u_\ell)\Big\}.
\end{eqnarray*}
Hence, $s=2$ is the support index of the ideal $I$. 
\end{Expl}

If one identifies every squarefree monomial $u\in S$ with its support, then if $I$ is a squarefree ideal of $S$ with $G(I)=\{u_1,\dots,u_p\}$, one obtains an equivalent definition of the support index $s$ defined above. More precisely, let
$\ell$ be the smallest number with the property that for all integers $1\le j_1< j_2 < \cdots < j_\ell\le p$, one has 
\[\lcm(u_{j_1}, u_{j_2}, \ldots, u_{j_\ell}) = \lcm(u_1,\dots,u_p).\]

Such an integer introduced in \cite{HPV} is called $\bcos_S(I)$. Hence, the support index $s$ of $I$ is equal to the integer $\bcos_S(I)-1$.  In \cite[Corollary 8]{CLT}, the authors proved that the projective dimension of a squarefree ideal $I$ of $S$ can be bounded by $\bcos_S(I)-1$. 

In \cite{HPV},  the notion of cosize of a squarefree ideal $I$ of $S$ with minimal system of generators given by $\{u_1,\dots,u_p\}$  has been introduced too. More precisely, let $w$ be the smallest number $\ell$ with the property that there exist $1\le j_1< j_2 < \cdots < j_\ell \le p$ such that 
\[\lcm(u_{j_1}, u_{j_2}, \ldots, u_{j_\ell}) = \lcm(u_1,\dots,u_p),\]
the number $$\deg \lcm(u_1,\dots,u_p) - w$$ is called cosize of $I$ and denoted by $\cosi(I)$. Moreover, in \cite[Corollary 1.4]{HPV}, the authors  proved that $\reg(I) \le \cosi(I) +1$.

Both the cited bounds were obtained by some results due to Lyubeznik \cite{GL} and applying the Alexander duality.
We will show how to compute the projective dimension and the regularity of squarefree ideals (thus, also for $t$--spread ideals) by using only set--theoretic operations. A fundamental tool will be the Taylor complex, which for each monomial ideal provides a graded free resolution, but which in general is not minimal \cite{JT}.

\begin{Thm}\label{thm:mainteorragtspread}
Let $I$ be a squarefree ideal of $S=K[x_1,\dots,x_n]$ with minimal generating set $G(I)=\{u_1,\dots,u_p\}$. Let $s$ be the support index of $I$, then
\begin{enumerate}
\item[\em(a)] $\pd(I) \le \min\{s,n\} =\min\{\bcos_S(I)-1, n\}$,
\item[\em(b)] $\reg(I) \le \cosi(I) +1$.
\end{enumerate}
\end{Thm}
\begin{proof} In order to simplify the notation, we set
\[
\Omega_i = \supp(u_i),\, i=1,\dots, p\]
and 
\[\Omega  = \bigcup_{i=1}^p\Omega_i.\]
Hence,
\begin{equation}\label{eq:A_0B_0sets}
\begin{aligned}
\A_0\ &= \big\{\Omega_1,\dots,\Omega_p\big\},\\
\A_i\ & = \Big\{\bigcup_{\ell=1}^{i+1}\Omega_{j_\ell}: \{j_1 < j_2 < \cdots < j_{i+1}\}\subseteq [p]\Big\},\, \mbox{for $i=1,\dots,p-1$}.
\end{aligned}
\end{equation}

To get the statements, we use the \emph{Taylor resolution}. Following the same notations as in \cite{JT}, to the sequence $\{u_1,\dots,u_p\}$  of monomial generators of $I$, we associate a complex $\mathbb{T}$ of free $S$--modules defined as follows: let $T_1$ be a free $S$--module with basis $\{e_1, \ldots, e_p\}$. Then
\begin{enumerate}
\item[--] $T_i\ =\ {\bigwedge}^iT_1$. More precisely, $T_i$ is a free $S$--module with basis the elements
$${\bf e}_F\ =\ e_{j_1}\wedge e_{j_2}\dots\wedge e_{j_i},\,\, F=\{j_1<j_2<\dots<j_i\}\subseteq[p];$$ 
\item[--] the differentials $\partial_i:T_i\rightarrow T_{i-1}$, for $i=1,\dots,p$, are defined by
$$
\partial_i({\bf e}_F)\ =\ \sum_{i\in F}(-1)^{\sigma(F,i)}\frac{\lcm(u_j:j\in F)}{\lcm(u_j:j\in F\setminus\{i\})}{\bf e}_{F\setminus\{i\}},
$$
where $\sigma(F,i)=\big|\{j\in F:j<i\}\big|$.
\end{enumerate}
To each basis element ${\bf e}_F$, with $F=\{j_1<j_2<\dots<j_i\}\subseteq[p]$, we assign degree 
\begin{equation}\label{eq1:deg}
\deg({\bf e}_F)=\deg(\lcm(u_j:j\in F))=|A|,
\end{equation}
with $A=\Omega_{j_1}\cup\Omega_{j_2}\cup\dots\cup\Omega_{j_i}\in\A_{i-1}$.
The differentials $\partial_i$ are homogeneous maps of graded free $S$--modules and so 
\begin{equation}\label{eq2:tay}
\mathbb{T}:0\rightarrow T_p\xrightarrow{\ \partial_{p}\ }T_{p-1}\xrightarrow{\ \partial_{p-1}\ }\cdots \xrightarrow{\ \partial_3\ }T_2 \xrightarrow{\ \partial_{2}\ }T_1\rightarrow I\rightarrow0
\end{equation}
is a graded free resolution of $I$ \cite[Theorem 7.1.1]{JT}. 
We can write $T_i=\bigoplus_{j}S(-j)^{\text{b}_{i,j}}$, for some positive integers $\text{b}_{i,j}$, 
and consider the unique minimal graded free $S$--resolution of $I$
$$
\FF: 0\rightarrow F_r\xrightarrow{\ d_r\ } F_{r-1}\xrightarrow{d_{r-1}}\cdots \xrightarrow{\ d_3\ }F_2\xrightarrow{\ d_2\ }F_1\rightarrow I\rightarrow0,
$$
with $F_i=\bigoplus_jS(-j)^{\beta_{i,j}}$, thus (Theorem \ref{thm:betaijlebij})
\begin{equation}\label{eq:mainteorragtspread}
\beta_{i,j}(I)\le\text{b}_{i,j}, \quad \mbox{for all $i$ and $j$}.
\end{equation}
(a) Consider the support index $s=\min\big\{i:  \mbox{ for all $A \in \A_i$}, A = \bigcup_{\ell=1}^p\Omega_\ell = \Omega \big\}$ of the squarefree ideal $I$, then
$$
\pd(I)=r \le s.
$$

Indeed, by the meaning of $s$, $T_s$ is a graded free $S$--module with a basis whose elements have degrees $\le \vert \Omega\vert$, whereas $T_{s+1}$ is a graded free $S$--module with a basis whose elements all have degrees equal to $\vert \Omega\vert$. The same reasoning holds for all $T_j$ with $s+1\leq j\leq p$. 

By (\ref{eq:mainteorragtspread}), $F_{s+2}$ is a graded free $S$--submodule of $T_{s+2}$. If $F_{s+2}\neq 0$, 
then, by the minimality of $\FF$, $\im(F_{s+2})\subseteq\mathfrak{m}F_{s+1}$, where $\mathfrak{m} = (x_1,x_2,\dots,x_n)$. Thus there exists a basis element of $F_{s+2}$ whose degree is greater than $\vert\Omega\vert$.  An absurd. Hence, $\pd(I)\le s$.

On the other hand, by Hilbert Syzygy's Theorem, $\pd(I)\le n$. Hence, $\pd(I)\le\min\{s,n\}$.\\
(b) From (\ref{eq:mainteorragtspread}), we have
\begin{align*}
\reg(I)&=\max \big\{j-i:\beta_{i,j}(I)\ne 0,\ \text{for some}\ i\ \text{and}\ j\big\}\\
&\le\max \big\{j-i:\text{b}_{i,j}\ne 0,\ \text{for some}\ i\ \text{and}\ j \big\}.
\end{align*}
From the Taylor resolution (\ref{eq2:tay}), one gets that the number 
$$\max \big\{j-i:\text{b}_{i,j}\ne 0,\ \text{for some}\ i\ \text{and}\ j\big\}$$ is determined by the \emph{first} step of the resolution $\mathbb{T}$ in which the free $S$--module has at least one direct summand $S(-j)$ with $j = \vert \Omega\vert$. More precisely, if $1\le \ell\le p$ is the smallest integer such that the free $S$--module $T_\ell = \bigoplus_{j}S(-j)^{\text{b}_{\ell,j}}$ has the property that there exists a direct summand $S$ shifted by $j$ with $j=\vert \Omega\vert$ (for some $j$), then
\[\max \big\{j-i:\text{b}_{i,j}\ne 0,\ \text{for some}\ i\ \text{and}\ j\big\} = \vert \Omega\vert - (\ell -1) =  \cosi(I) +1.\]
The assertion follows.
\end{proof}

Let us consider Example \ref{ex:sets}. One can observe that the ideal $I$ has the following minimal graded free $S$--resolution
\[\FF: 0\rightarrow S(-8) \rightarrow S(-5)\oplus S(-6)^2 \rightarrow S(-2)\oplus S(-3)\oplus S(-4)\rightarrow I\rightarrow0.\]
Thus, $\pd(I) = 2 = s$ and $\reg(I)= 6 = \vert \bigcup_{j=1}^3\supp(u_j) \vert -2 = 8-2$.

\begin{Rem}\label{rem:ossregtspread}\rm
From Theorem \ref{thm:mainteorragtspread}, if a squarefree monomial ideal $I$ has $p$ generators, then  
$s\le p-1$.
\end{Rem}

The following example illustrates Theorem \ref{thm:mainteorragtspread}.

\begin{Expl}\rm
Let $S=\QQ[x_1,\ldots,x_9]$, and let $I$ be the following squarefree ideal
\[
I=(x_1x_4, x_1x_3x_8, x_2x_4x_6, x_1x_3x_5x_7x_9).
\]
By using \emph{Macaulay2} \cite{GDS}, the Taylor resolution of $I$ is the following one:\\

\noindent
\resizebox{\textwidth}{!}{
\begin{math}
\begin{aligned}
0\,\xrightarrow{}\,\underset{}{S(-9)}\,\xrightarrow{
\left(\!\begin{smallmatrix}
-x_{5}x_{7}x_{9}\\
-x_{2}x_{6}\\
x_{8}\\
1\\
\end{smallmatrix}\!\right)}\,
\underset{}{S(-6)\oplus S(-7)\oplus S(-8)\oplus S(-9)}\,\xrightarrow{
\left(\!\begin{smallmatrix}
x_{3}x_{8}&0&x_{3}x_{5}x_{7}x_{9}&0\\
-x_{2}x_{6}&x_{5}x_{7}x_{9}&0&0\\
1&0&0&x_{5}x_{7}x_{9}\\
0&-x_{8}&-x_{2}x_{6}&0\\
0&x_{4}&0&x_{2}x_{4}x_{6}\\
0&0&1&-x_{8}\\
\end{smallmatrix}\!\right)}\\
\underset{}{S(-4)^2\oplus S(-6)^3\oplus S(-8)}\,\xrightarrow{
\left(\!\begin{smallmatrix}
-x_{2}x_{6}&-x_{3}x_{8}&0&-x_{3}x_{5}x_{7}x_{9}&0&0\\
x_{1}&0&-x_{1}x_{3}x_{8}&0&0&-x_{1}x_{3}x_{5}x_{7}x_{9}\\
0&x_{4}&x_{2}x_{4}x_{6}&0&-x_{5}x_{7}x_{9}&0\\
0&0&0&x_{4}&x_{8}&x_{2}x_{4}x_{6}\\
\end{smallmatrix}\!\right)}\\
\underset{}{S(-2)\oplus S(-3)^2\oplus S(-5)}
\xrightarrow{
\left(\!\begin{smallmatrix}
x_{1}x_{4}&x_{2}x_{4}x_{6}&x_{1}x_{3}x_{8}&x_{1}x_{3}x_{5}x_{7}x_{9}\\
\end{smallmatrix}\!\right)
}\,
\underset{}{I}\,\rightarrow\,0.\\
\end{aligned}
\end{math}
}\\

Moreover, the sets which allow us to compute the support index $s$ of $I$ are the followings: 
\begin{eqnarray*}
\A_0&=& \big\{\{1, 4\}, \{1, 3, 8\}, \{2, 4, 6\}, \{1, 3, 5, 7, 9\}\big\},\\
\A_1&=&\big\{\{1,2,4,6\}, \{1,2,3,4,6,8\}, \{1,2,3,4,5,6,7,9\}, \{1,3,4,8\},\\
 & &\ \ \{1,3,4,5,7,9\}, \{1,3,5,7,8,9\}\big\},\\
\A_2&=&\big\{\{1,2,3,4,5,6,7,8,9\}, \{1,2,3,4,5,6,7,9\}, \{1,3,4,5,7,8,9\},\\
 & &\ \ \{1,2,3,4,6,8\}\big\},\\
\A_3&=& \big\{\{1,2,3,4,5,6,7,8,9\}\big\}.
\end{eqnarray*}

Hence, $s=3$ and we obtain that $\pd(I)\le 3$ and $\reg(I)\le 9-2=7$. Indeed, in such a case, $\pd(I)=2$ and $\reg(I)=5$.
\end{Expl}

We close the section with an example of a classical squarefree monomial ideal for which the bounds in Theorem \ref{thm:mainteorragtspread} are optimal in some special cases. More precisely, we consider the squarefree Veronese ideal $I_{n, d}$ of degree $d\ge 1$ of $S=K[x_1,\dots,x_n]$. It is the squarefree monomial ideal of $S$ generated by all squarefree monomials of degree $d$. 

If $F\subseteq[n]$, let ${\bf x}_F$ the monomial of  $S=K[x_1,\dots,x_n]$ defined as follows:

\[{\bf x}_F\ =\ \prod_{i\in F}x_i.\]
\begin{Expl}\label{ex:VeroneseSquarefree} \em Let us consider the squarefree Veronese ideal $I_{n,d}$ of degree $d\ge 1$ in $S=K[x_1,\dots,x_n]$. 
It is well--known that $\pd(I_{n,d})=n-d$ and $\reg(I_{n,d})=d$. 

Let us consider the case $d=n-1$. It is clear that in such a case $I=I_{n,n-1}$ is the squarefree monomial ideal of $S$ generated by all monomials ${\bf x}_{[n]}/x_i$, $i=1,\dots,n$. Moreover, $\vert G(I)\vert = n$, $\pd(I)=1$ and $\reg(I)=n-1$.

Now we want to compute the support index $s$ of $I$. Using our methods, if $\Omega_i=\supp({\bf x}_{[n]}/x_i)=[n]\setminus\{i\}$, for all $i=1,\dots,n$, then 
\begin{align*}
\Omega&=\bigcup_{i=1}^n\Omega_i=[n],\\
\mathcal{A}_0&= \big\{[n]\setminus\{i\}:i=1,\dots,n\big\},\\
\mathcal{A}_1&= \Big\{([n]\setminus\{i\})\cup([n]\setminus\{j\}):i,j=1,\dots,n,i\ne j\Big\}=\big\{[n]\big\}.
\end{align*}
Thus, the support index of $I$ is $s=1$ and the bound given in Theorem \ref{thm:mainteorragtspread} for the projective dimension is reached. 

Let us examine the Castelnuovo--Mumford regularity of $I$. One has
\[\vert \Omega\vert - 1= n-1 =\reg(I)\]
Hence, the bound stated in Theorem \ref{thm:mainteorragtspread} for the regularity is optimal. \\

Now, let consider the general case, \emph{i.e.}, $I_{n,d}$, $d\ge 1$.\\
Let $\A_i$ be the set of all  unions of $i+1$ supports of the monomials of $G(I_{n,d})$.
We prove that the support index $s$ of $I_{n,d}$ is
\begin{equation}\label{eq:supportIndexSquarefreeVeronese}
s = \binom{n-1}{d}.
\end{equation}
Indeed, there are $\binom{n-1}{d}$ subsets of the set $[n-1]$ with cardinality $d$, and their union is the set $[n-1]$. Thus $[n-1]\in\A_{\binom{n-1}{d}-1}$ and $s\ge \binom{n-1}{d}$. \\
On the other hand, the set $\A_{\binom{n-1}{d}}$ consists of all the unions of $\binom{n-1}{d}+1$ supports of distinct monomial generators of $I_{n,d}$. \\
Suppose that at least one of these unions is not $[n]$, \emph{i.e.}, $s>\binom{n-1}{d}$. Hence, there exist $\binom{n-1}{d}+1$ distinct monomials $u_1,\dots,u_{\binom{n-1}{d}+1}\in G(I_{n,d})$ and an integer $i\in[n]$ such that
$$
\bigcup_{i=1}^{\binom{n-1}{d}+1}\supp(u_i)=[n]\setminus\{i\}.
$$
Without loss of generality,  
we can assume $i=n$. Thus, there exist $\binom{n-1}{d}+1$ different subsets of cardinality $d$ of the set $[n-1]$. An absurd. Therefore $s=\binom{n-1}{d}$.\\
Finally, we have $\pd(I_{n,d})=n-d\le\binom{n-1}{d}$ and the inequality in Theorem \ref{thm:mainteorragtspread} is verified. One can observe that the equality holds if and only if $d=1,n-1,n$. 
\end{Expl}

\section{Regularity of $t$--spread ideals}\label{sec:3}

In this section, we analyze the Castelnuovo--Mumford regularity of $t$--spread ideals, $t\ge 1$. Moreover, we identify a special class of $t$--spread ideals for which the bounds given in Theorem \ref{thm:mainteorragtspread} are reached.\\

Firstly, we rewrite a result on squarefree monomial ideals generated by a squarefree monomial regular sequence \cite{Ei, YP}. Our techniques are completely different from the classical ones and manipulate sets of squarefree monomials.

\begin{Thm}\label{thm:teorpart}
Let $n$ be a positive integer and let $I$ be a squarefree ideal of $S=K[x_1,\dots,x_n]$ generated by a squarefree monomial regular sequence. 
Set $\A_0=\{\supp(u): u \in G(I)\}$.
Then
\begin{enumerate}
\item[\em(a)] $\pd(I) =\vert \A_0\vert -1$,
\item[\em(b)] $\reg(I) = \Big\vert \bigcup_{u\in G(I)}\supp (u)\Big\vert -(|\A_0| -1)$.
\end{enumerate}
\end{Thm}
\begin{proof}
Let $G(I) = \{u_1, \ldots, u_p\}$. Set $\Omega_i = \supp (u_i)$, for $i=1, \ldots, p$, $\Omega = \bigcup_{i=1}^p\Omega_i$ and $\A_i=\big\{\Omega_{j_1}\cup\dots\cup \Omega_{j_i}\cup \Omega_{j_{i+1}}: \{j_1 < j_2 < \cdots < j_{i+1}\}\subseteq [p]\big\}$, for $i=1,\dots,p-1$. Since $I$ is generated by a squarefree monomial regular sequence, then $\A_0$ is a partition of the set $\Omega$.
Moreover, $\vert \A_0 \vert = p$.\\
(a) Since the sequence ${\bf u}=u_1,\dots,u_p$ is a regular sequence on $S$, by \cite[Theorem A.3.4]{JT} the Koszul complex $K_{_{\text{\large$\boldsymbol{\cdot}$}}}({\bf u};S)$ is a minimal free $S$--resolution of $S/I=S/({\bf u})$.\\
Let $F$ be a free $S$--module with basis $e_1,\dots,e_p$ such that each element $e_i$ has degree $\deg(e_i)=\deg(u_i)=|\Omega_i|$. 
Hence, a minimal graded free $S$--resolution of $I$ is
$$
\mathbb{F}:0\rightarrow{\bigwedge}^p F\rightarrow{\bigwedge}^{p-1}F\rightarrow\cdots\rightarrow{\bigwedge}^1F\rightarrow I\rightarrow0
$$
and  $\pd(I) =p-1 = \vert \A_0\vert -1$. \\
(b) For each $i=0,\dots,p-1$, a basis of the free $S$--module ${\bigwedge}^{i+1}F$ in the resolution $\mathbb{F}$ is given by the wedge products ${\bf e}_T=e_{j_1}\wedge\dots\wedge e_{j_i}\wedge e_{j_{i+1}}$, for all $T=\{j_1<j_2<\dots<j_{i+1}\}\subseteq [p]$. The basis element ${\bf e}_T$ has degree
$$
\deg({\bf e}_T)=\sum_{\ell=1}^{i+1}\deg(e_{j_\ell})=\sum_{\ell=1}^{i+1}|\Omega_{j_{\ell}}|=\bigg|\bigcup_{\ell=1}^{i+1}\Omega_{j_\ell}\bigg|,
$$
where the last equality follows from the fact that $\Omega_{j_1},\dots,\Omega_{j_{i+1}}$ are pairwise disjoint. Thus
\begin{equation}\label{eq:RegPrecisaUnioni}
\reg(I) = \vert \Omega\vert -(p-1) = |\Omega_1\cup\dots\cup \Omega_p| -(p-1).
\end{equation}
\end{proof}

\begin{Expl}
\rm Let $S=\QQ[x_1, \ldots, x_{8}]$ and consider the following squarefree ideal $I=(x_8,x_1x_2, x_3x_4x_5x_7)$ of $S$. Set $\Omega_1 =\supp(x_8)=\{8\}$, $\Omega_2=\supp(x_1x_2)=\{1,2\}$ and $\Omega_3=\supp(x_3x_4x_5x_{7})=\{3,4,5,7\}$. Let $\A_0=\{\Omega_1, \Omega_2, \Omega_3\}$. With the same notations as in Section \ref{sec:2}, we have
\begin{eqnarray*}
	\A_0&=&   \big\{\{1,2\}, \{8\}, \{3,4,5,7\}\big\}, \\
	\A_1&=&   \big\{\{1,2,8\}, \{3,4,5,7,8\},\{1,2,3,4,5,7\}\big\},\\
	\A_2&=&   \big\{\{1,2,3,4,5,7,8\}\big\}.
\end{eqnarray*}
Hence, $s=2$ is the support index of the ideal $I$. In such a case, $s=\vert G(I)\vert -1$. By Theorem \ref{thm:teorpart}, $\pd(I)=|\A_0|-1=2$ and $\reg(I)=|\Omega_1\cup \Omega_2\cup \Omega_3|-(|\A_0|-1)=7-(3-1)=5$.
\end{Expl}

Now, we are in position to prove the bound for the regularity of $t$--spread ideals. We assume $n \ge t$. Indeed, if $n < t$, then the only t--spread monomials are simply variables.
\begin{Thm}\label{thm:maximumregtspread}
Let $t\ge 1$ and $I$ a $t$--spread ideal of $S=K[x_1,\dots,x_n]$ with $n\ge t$. Then $$
\reg(I)\ \le\ n-(t-1).
$$
\end{Thm}
\begin{proof} Let us denote by $\mathcal T$ the class of all $t$--spread ideals of $S=K[x_1,\dots,x_n]$. 
The proof consists of two steps.\\
\textsc {Step 1.} There exists a $t$--spread ideal $J$ of $S$ such that $\reg (J) \ge \reg(I)$, for all $I\in \mathcal T$.\\
\textsc {Step 2.} $\reg(J) = n-(t-1)$.\\\\
\textsc {Step 1.} 
Let us consider the ideal $J\in \mathcal T$ such that $\bigcup_{u\in G(J)}\supp(u) = [n]$ and with the minimum possible number of generators.  \\
\emph{Claim 1.} $\vert G(J)\vert = t$.\\
To prove this, we use the \textit{least criminal} technique \cite[Proposition 1.1]{JJR}. 
For $n=t$, the minimum possible number of generators of a $t$--spread ideal $H$ for which the union of the supports is $[n]=[t]$ is $t$, \emph{i.e.},  $G(H) = \{x_1,x_2,\dots,x_t\}$.\\
Let  $n>t$. Suppose that, for some $n$, there exists a $t$--spread ideal $H$ with less than $t$ generators. Let $G(H)=\{u_1,\dots,u_p\}$ with $p<t$ and $\bigcup_{i=1}^p\supp(u_i)=[n]$. Let $\bar n$ be the minimal possible integer which satisfies such a property. We have $\bar n\ge 1+t$. For all $i=1,\dots,p$, let us consider the monomials
\begin{equation}\label{eq:formvi}
v_i\ =\ \begin{cases}
\hfil u_i&\text{if}\ \bar n\notin\supp(u_i),\\
\hfil u_i/x_{\bar n}&\text{if}\ \bar n\in\supp(u_i).
\end{cases}
\end{equation}
The ideal $\overline{H}=(v_1,\dots,v_p)\subseteq K[x_1,\dots,x_{\bar n-1}]$ is such that $\bigcup_{i=1}^p\supp(v_i)=[\bar n-1]$. It is an absurd by the meaning of $\bar n$. The claim follows.\\\\
\emph{Claim 2.} If $u_1,u_2,\dots,u_t$ are $t$--spread monomials such that $\bigcup_{i=1}^t\supp(u_i)=[n]$, then
\begin{equation}\label{eq:decompind}
\supp(u_i)\ =\ \big\{\ell\in[n]:\ell\equiv n-(i-1)\ (\text{mod}\ t)\big\},
\end{equation}
for $i=1,\dots,t$.\\
We proceed by induction on $n\ge t$. For $n=t$, it follows easily. \\
Suppose $n\ge 1+t$.  
We have to show that condition (\ref{eq:decompind}) holds for all $i=1,\dots,t$. Let us consider the $t$--spread monomials $v_i$ described in (\ref{eq:formvi}). The $t$--spread monomials $v_1,\dots,v_t$ are such that $\supp(v_1)\cup\dots\cup\supp(v_t)=[n-1]$. By induction, after a suitable reindexing, we may suppose that
$$
\supp(v_i)\ =\ \big\{\ell\in[n-1]:\ell\equiv n-1-(i-1)\ (\text{mod}\ t)\big\},
$$
for $i=1,\dots,t$. For all $i$, we have $\supp(u_i)=\supp(v_i)$ or $\supp(u_i)=\supp(v_i)\cup\{n\}$. The second possibility occurs if and only if $u_i=v_ix_n$ is $t$--spread, \emph{i.e.}, if and only if $\max(v_i)\le n-t$. This is true if and only if $i=t$. Therefore,
\begin{align*}
\supp(u_t)&\ =\ \supp(v_t)\cup\{n\}\\&\ =\ \big\{\ell\in[n-1]:\ell\equiv n-1-(t-1)\ (\text{mod}\ t)\big\}\cup\{n\}\\
&\ =\ \big\{\ell\in[n-1]:\ell\equiv n\ (\text{mod}\ t)\big\}\cup\{n\}\\
&\ =\ \big\{\ell\in[n]:\ell\equiv n\ (\text{mod}\ t)\big\},
\end{align*}
and $\supp(u_i)=\supp(v_i)$, for $i=1,\dots,t-1$. After a suitable reindexing,  the property in (\ref{eq:decompind}) is true for $i=1,\dots,t$. The claim follows. \\
Finally, from Theorem \ref{thm:mainteorragtspread}, it is clear that $J=(u_1,u_2,\dots,u_t)$ is the ideal of $\mathcal T$ with maximum regularity.\\\\
\textsc{Step 2.} 
Let us consider the ideal $J$ constructed in \textsc{Step 1}.
Setting, $\Omega_i = \supp(u_i)$, since the sets $\Omega_i$ are pairwise disjoint, \emph{i.e.}, $u_1, \ldots, u_p$ is a regular sequence, 
by Theorem \ref{thm:teorpart}, we have $$\reg(J)=|\Omega|-(t-1)=|[n]|-(t-1)=n-(t-1).$$
\end{proof}

Recall that the \textit{Alexander dual} of a simplicial complex $\Delta$ is the simplicial complex: $\Delta^\vee=\big\{[n]\smallsetminus F:F\notin\Delta\big\}$. 
\begin{Cor}
	Let $I$ be a $t$--spread monomial ideal of $S=K[x_1,\dots,x_n]$, and let $\Delta$ be the unique simplicial complex on $[n]$ such that $I_\Delta=I$. If $\Delta^\vee$ is the Alexander dual of $\Delta$, then $\pd(I_{\Delta^\vee})\le n-t$ and $\depth(I_{\Delta^\vee})\ge t$.
\end{Cor}
\begin{proof}
	By \cite{TER}, $\reg(I)=\reg(I_\Delta)=\pd(S/I_{\Delta^\vee})=\pd(I_{\Delta^\vee})+1$. Hence $\pd(I_{\Delta^\vee})\le n-t$, by Theorem \ref{thm:maximumregtspread}. Moreover, from the Auslander--Buchsbaum formula, we have
	$$
	\pd(I_{\Delta^\vee})\ =\ \depth(S)-\depth(I_{\Delta^\vee})\ =\ n-\depth(I_{\Delta^\vee})\ \le\ n-t.
	$$
Finally,  $\depth(I_{\Delta^\vee})\ge t$.
\end{proof}
\vspace{0,5cm}
%

We close this section discussing the $t$--spread ideal $J$ introduced in the proof of Theorem \ref{thm:maximumregtspread} and some related properties. As we have pointed out, $J$ is a special $t$--spread ideal which \emph{reachs} the upper bound for the Castelnuovo--Mumford regularity given in Theorem \ref{thm:maximumregtspread}. \\

Given a pair of positive integers $(n, t)$, with $n\ge t$, let us consider the $t$--spread monomials defined in (\ref{eq:decompind}), \emph{i.e.}, 
\begin{equation}\label{eq:pascal}
{\bf x}_{n,t,i}\ =\ \prod_{\substack{j\equiv i\ (\text{mod}\ t) \\ 1\le j\le n}}x_j, \qquad i\in[t]=\{1,\dots,t\}.
\end{equation}

Let us define the following class of monomial ideals.
\begin{Def}\label{def:Pascal}\em
A $t$--spread ideal of $S$ is called a \textit{Pascal ideal of type} $(n,t)$ if it is generated by the monomials in (\ref{eq:pascal}).
\end{Def}

We will denote it by $I_{\textup{Pasc},n,t}$.

\begin{Expl}\rm Let $(n,t)=(10,3)$, then the Pascal ideal of type $(10,3)$ is
\begin{align*}
I_{\text{Pasc},10,3}\ =\ ({\bf x}_{10,3,1},{\bf x}_{10,3,2},{\bf x}_{10,3,3})=(x_1x_4x_7x_{10},x_2x_5x_8,x_3x_6x_9).
\end{align*}
\end{Expl}

\vspace{0,3cm}

The name of such an ideal is justified by the next remark.
\begin{Rem} \label{rem:pasc}\rm
Let $I=I_{\text{Pasc},n,t}$ be the Pascal ideal of type $(n,t)$. 
From the structure of $I$, the minimal graded free $S$--resolution of $S/I$ is given by the Koszul complex 
$K_{_{\text{\large$\boldsymbol{\cdot}$}}}({\bf x}_{n,t};S)$ attached to the regular sequence ${\bf x}_{n,t}$ $= {\bf x}_{n,t,1},$ ${\bf x}_{n,t,2},$ $\dots, $ ${\bf x}_{n,t,t}$ generating $I$, \emph{i.e.},
$$
K_{_{\text{\large$\boldsymbol{\cdot}$}}}({\bf x}_{n,t};S):0\rightarrow{\bigwedge}^t F\rightarrow{\bigwedge}^{t-1}F\rightarrow\dots\rightarrow{\bigwedge}^0F\rightarrow S/I \rightarrow0,
$$
where $F$ is the free $S$--module with basis $e_1,\dots,e_t$, and $\deg(e_i)=\deg({\bf x}_{n,t,i})$, for all $i$. Note that
$\pd(S/I)=t$, and
$$
\beta_i\big(S/I\big)\ =\ \binom{t}{i},\,\,\ \text{for all}\ i=0,\ldots,t,
$$
\emph{i.e.}, the total Betti numbers of $S/I$ are the entries of the $t$--th row of the Pascal triangle.
\end{Rem}

The next remark will be useful in the sequel.
\begin{Rem} \rm Let $n,t\ge1$. The maximum degree $\ell$ of a $t$--spread monomial in $S=K[x_1,\dots,x_n]$ can be determined as follows:
\begin{equation}\label{eq:max}
\begin{aligned}
\ell\ &=\ \max\big\{d\ge0:M_{n,d,t}\ne\emptyset\big\}\\
&=\ \max\Big\{d\ge0:\binom{n-(d-1)(t-1)}{d}\ne0\Big\}\\
&=\ \max\big\{d\ge0:n-(d-1)(t-1)\ge d\big\}\\
&=\ \left\lfloor\frac{n-1}{t}\right\rfloor+1.
\end{aligned}
\end{equation}
\end{Rem}\medskip

\begin{Prop} \label{prop:pasc}
Let $(n,t)$ be a pair of positive integers with $n\ge t$. Let $I=I_{\textup{Pasc},n,t}$ be the Pascal ideal of type $(n,t)$. Then 
\begin{enumerate}
\item[\em(a)] $S/I$ is a Cohen--Macaulay ring;
\item[\em(b)] $I$ has only one extremal Betti number. In particular, $\pd(I)=t-1,\ \reg(I)=\depth(I)=n-(t-1)$;
\end{enumerate}
Moreover, if $i\in [t]$ is such that $n\equiv i\ (\textup{mod}\ t)$, then 
\begin{enumerate}
\item[\em(c)] the $f_t$--vector of $I$ is
\vspace{-0.1em}$$
f_t(I)\ =\ \big(f_{t,0},f_{t,1},\ldots,f_{t,\lfloor\frac{n-1}{t}\rfloor+1}\big),
$$
with\vspace{-0.5em}
\begin{equation}\label{eq:ftvectorpasc}
f_{t,j}\ =\ \begin{cases}
\binom{n-(j-1)(t-1)}{j}&\text{if}\ j=0,\dots,\lfloor\frac{n-1}{t}\rfloor-1,\\[3pt]
\binom{n-(j-1)(t-1)}{j}-(t-i)&\text{if}\ j=\lfloor\frac{n-1}{t}\rfloor,\\
\binom{n-(j-1)(t-1)}{j}-i&\text{if}\ j=\lfloor\frac{n-1}{t}\rfloor+1;
\end{cases}
\end{equation}
\item[\em(d)] the Hilbert series of $S/I$ is
$$
\textup{Hilb}_{S/I}(z)\ =\ \frac{{(1+z+\dots+z^{\lfloor\frac{n-1}{t}\rfloor})}^{i}{(1+z+\dots+z^{\lfloor\frac{n-1}{t}\rfloor-1})}^{t-i}}{{(1-z)}^{n-t}};
$$
\item[\em(e)] there exists a $t$--spread lexsegment ideal $L$ of $S$ such that $f_t(L)=f_t(I)$ if and only if $i\in \{1, t-1, t\}$.
\end{enumerate}
\end{Prop}
\begin{proof}
If $n=t$, then $I=\mathfrak{m}=(x_1,\dots,x_t)$ and all the statements follow.\\\\
Let $n\ge 1+t$. \\
(a) Since $I$ is generated by a regular sequence, then $I$ is  complete intersection and the assertion follows.\\
(b) From (a), $I$ has a unique extremal Betti number \cite[Theorem 2.16]{HHO}. Furthermore,
$\pd(I)=t-1$ (Theorem \ref{thm:teorpart}), and from the Auslander--Buchsbaum formula and Theorem \ref{thm:maximumregtspread}, 
we have $$\depth(I)=n-\pd(I)=n-(t-1) = \reg(I).$$ 
(c) Let $i\in [t]$ such that $n\equiv i\ (\textup{mod}\ t)$. We can write $n=i+kt$, with $k=\frac{n-i}{t}$.

Since $0\le i-1\le t-1$, we have
$$
k=\frac{kt}{t}+\left\lfloor\frac{i-1}{t}\right\rfloor=\left\lfloor\frac{i-1+kt}{t}\right\rfloor=\left\lfloor\frac{n-1}{t}\right\rfloor.
$$
Hence
\begin{align}\label{eq:degxnti}
\nonumber \deg({\bf x}_{n,t,j})\ &=\ \begin{cases}
\big|\{j,j+t,\dots,j+kt\}\big|&\text{if}\,\, \,j=1,\dots,i,\\
\big|\{j,j+t,\dots,j+(k-1)t\}\big|&\text{if}\,\,\, j=i+1,\dots,t.
\end{cases}\\
&=\ \begin{cases}
\left\lfloor\frac{n-1}{t}\right\rfloor+1&\text{if}\ j=1,\dots,i,\\
\left\lfloor\frac{n-1}{t}\right\rfloor&\text{if}\ j=i+1,\dots,t.
\end{cases}
\end{align}\smallskip
One can observe that $\left\lfloor\frac{n-1}{t}\right\rfloor+1$ is the maximum degree of a monomial generator of the Pascal ideal $I$. Hence, setting $r=|[I_{k+1}]_t|$ and $s=|[I_{k}]_t|$, then $r+s=t=\left|G(I)\right|$. 
Moreover, by the pairwise disjointness of the supports of these monomials, we have that $r(k+1)+s k=n$ which is the number of the indeterminates of $S$. Solving the linear system
\[r+s=t,\quad r(k+1)+s k=n,\]
one obtains $s=t-r$ and $r+kt=n$. Recalling that, by hypothesis,  $n=i+k t$, with $1\leq i\leq t$, one has that $r=i$ and $s=t-i$.\\
Since $I_\ell=0$, for $0\le\ell\le\left\lfloor\frac{n-1}{t}\right\rfloor-1$, and $\Shad_t(\{{\bf x}_{n,t,i+1},\dots,{\bf x}_{n,t,t}\})=\emptyset$, 
(c) follows. \\
(d) Since $I=({\bf x}_{n,t})$ and ${\bf x}_{n,t}={\bf x}_{n,t,1},\ldots,{\bf x}_{n,t,t}$ is a regular sequence, from (\ref{eq:degxnti}) and \cite[Problem 6.2]{JT}, we have 
\begin{align*}
\textup{Hilb}_{S/I}(z)\ &=\ \frac{\prod_{u\in G(I)}(1+z+\ldots+z^{\deg(u)-1})}{(1-z)^{n-|G(I)|}}\\ &=\ \frac{{(1+z+\ldots+z^{\lfloor\frac{n-1}{t}\rfloor})}^{i}{(1+z+\ldots+z^{\lfloor\frac{n-1}{t}\rfloor-1})}^{t-i}}{{(1-z)}^{n-t}},
\end{align*}
and (d) holds.\\
(e) Assume that $M_{n,k,t}$ is endowed with the squarefree lex order.  Let $i\in [t]$ such that $n=i+kt$. We have already noted that $k=\left\lfloor\frac{n-1}{t}\right\rfloor$.\\
Our aim is to determine a $t$--spread lexsegment ideal $L$ of $S$ such that $f_t(L)=f_t(I)$. \\
Assume $i=t$, then
\[I = \Big(\prod_{j=0}^{k}x_{1+jt},\prod_{j=0}^{k}x_{2+jt},\dots,\prod_{j=0}^{k}x_{t+jt}\Big),\]
and 
\[L=\ \Big(\prod_{j=0}^{k}x_{1+jt},\Big(\prod_{j=0}^{k-1}x_{1+jt}\Big)x_{2+kt},\dots,\Big(\prod_{j=0}^{k-1}x_{1+jt}\Big)x_{t+kt}\Big)\]
 is the $t$--spread lexsegment ideal we are looking for. \\
In fact, the Pascal ideal $I$ has exactly $t$ generators of degree $k+1=\left\lfloor\frac{n-1}{t}\right\rfloor+1$. Hence, in order to get the desired $t$--spread lexsegment ideal, we need to take the greatest $t$--spread monomials of $M_{n,k+1,t}$ with respect to the squarefree lex order.\\\\
Now, assume $1\le i\le t-1$. By (\ref{eq:degxnti}), we have that
\[I= \Big(\prod_{j=0}^{k}x_{1+jt},\dots,\prod_{j=0}^{k}x_{i+jt},\prod_{j=0}^{k-1}x_{(i+1)+jt},\dots,\prod_{j=0}^{k-1}x_{t+jt}\Big).\]
A possible candidate for $L$ is the following $t$--spread lexsegment ideal of $S$:
\[L= \Big(\prod_{j=0}^{k-1}x_{1+jt},\Big(\prod_{j=0}^{k-2}x_{1+jt}\Big)x_{2+(k-1)t},\dots,\Big(\prod_{j=0}^{k-2}x_{1+jt}\Big)x_{t-i+(k-1)t}\Big).\]
As we have observed previously, $I$ has $i$ generators of degree $k+1$ and $t-i$ generators of degree $k$. In order to have $f_t(I) = f_t(L)$, we must take into account the largest $t-i$ monomials of $M_{n,k,t}$. Since $t-i<n$, then we must consider all monomials in $M_{n,k,t}$ obtained starting from $\prod_{j=0}^{k-1}x_{1+jt}$ which is the largest monomial of $M_{n,k,t}$. Then one have to fix the first $k-1$ indeterminates, 
and replace $x_{1+(k-1)t}$ with $x_q$ for $q\in \{2+(k-1)t, \ldots, t-i+(k-1)t\}$. One can observe that at least $i$ monomials of degree $k+1$ belong to the shadow of the ones of degree $k$ already built. Proceeding in such a way, we cover all the minimal generators of the Pascal ideal.\\\\
Let us analyze the conditions under which  $f_t(I) = f_t(L)$. \\
As we have said so far, it is clear that $f_{t,j}(L)=f_{t,j}(I)$, for all $j=1,\dots,k=\left\lfloor\frac{n-1}{t}\right\rfloor$. \\
Setting $\mathcal{V}=\Shad_t(L_k)$, by (d) and by the definition of $f_t$--vector,
\small\begin{align}\label{eq:fvecpasc}
 \nonumber f_{t,k+1}(I)\ &=\ \binom{n-k(t-1)}{k+1}-i,\\
&\\
\nonumber f_{t,k+1}(L)\ &=\ \binom{n-k(t-1)}{k+1}-|\mathcal{V}|.
\end{align}\normalsize
One can quickly verify that $\mathcal{V}$ consists of the following monomials:\\
\begin{align*}
\Big(\prod_{j=0}^{k-2}x_{1+jt}\Big)x_{1+(k-1)t}x_{\ell+kt},&&& \ell = 1, \ldots, i,\\
\Big(\prod_{j=0}^{k-2}x_{1+jt}\Big)x_{2+(k-1)t}x_{\ell+kt},&&& \ell  = 2, \ldots, i, \\
\vdots\ \ \ \ \ \ \ \ \ \ \ &&&\ \ \ \ \ \ \ \ \vdots\\
\Big(\prod_{j=0}^{k-2}x_{1+jt}\Big)x_{t-i+(k-1)t}x_{\ell+kt},&&& \ell = t-i, \ldots, i. 
\end{align*}\normalsize\\
Therefore
\begin{equation}\label{|mathcal:L|equ}
|\mathcal{V}|\ =\ \begin{cases}
\sum\limits_{j=0}^{t-i-1}(i-j)&\text{if}\ t-i\le i,\\
&\\
\sum\limits_{j=0}^{i-1}(i-j)&\text{if}\ t-i>i.
\end{cases}
\end{equation}\normalsize
\vspace{0,3cm}
\par\noindent
From (\ref{eq:fvecpasc}), 
$f_{t,k+1}(I)=f_{t,k+1}(L)$ if and only if $|\mathcal{V}|=i$.
If $t-i\le i$, from (\ref{|mathcal:L|equ}), $|\mathcal{V}|=i$ if and only if $t-i-1=0$, \emph{i.e.}, $i=t-1$. Otherwise, if $t-i>i$, then $|\mathcal{V}|=i$ if and only if $i=1$. 
Finally, (e) holds.
\end{proof}
\begin{Rem} \rm We have pointed out that if $I$ is a $t$--spread ideal, then $I$ may not have a $t$--spread lex ideal with the same $f_t$--vector \cite{CAC}, whereas if $I$ is a $t$--spread strongly stable ideal, then there exists a unique $t$--lex ideal, $I^{t, \lex}$, with the same $f_t$--vector as $I$ \cite[Theorem 2.1]{CAC}. The Pascal ideal $I=I_{\text{Pasc},n,t}$ is not a $t$--spread strongly stable ideal but under certain condition (Proposition \ref{prop:pasc}, (e)) satisfies such a property. Indeed, with the same notation as in Proposition \ref{prop:pasc}, $L = I^{t, \lex}$ in the sense of \cite{CAC}.
\end{Rem}
\begin{Expl}\rm
Consider $(n,t)=(10,3)$, then
\begin{align*}
I=I_{\text{Pasc},10,3}\ =\ ({\bf x}_{10,3,1},{\bf x}_{10,3,2},{\bf x}_{10,3,3})=(x_1x_4x_7x_{10},x_2x_5x_8,x_3x_6x_9).
\end{align*}
By \emph{Macaulay2}, the Betti diagram of $S/I$ is
\begin{center}
	\begin{tabular}{cccccc}
		&&0&1&2&3\\ 
		\text{Tot}&:&1&3&3&1\\ \hline 
		0&:&1&-&-&-\\
		1&:&-&-&-&-\\
		2&:&-&2&-&-\\
		3&:&-&1&-&-\\
		4&:&-&-&1&-\\
		5&:&-&-&2&-\\
		6&:&-&-&-&-\\
		7&:&-&-&-&1\\
	\end{tabular}
\end{center}\medskip
One has $\pd(S/I)=t=3$, $\reg(S/I)=7$, $\pd(I)=t-1=2$, and $\reg(I)=n-(t-1)=8$.

Proceeding as in Proposition \ref{prop:pasc}, there exists a $t$--spread lexsegment ideal $L$ of $S$ such that $f_t(I)=f_t(L)$. It is
$L =(x_1x_4x_7,x_1x_4x_8) = I^{t, \lex}$.
Indeed, $i=1$ and condition (e) in Proposition \ref{prop:pasc} is verified.
\end{Expl}

\section{Some Applications}\label{sec:appl}

In this section we consider some applications of the previous results. In particular, we  analyze the regularity of $t$--spread ideals equigenerated in degree two by tools from the graph theory.\\

A \emph{simple graph} $G$ is an ordered pair of disjoint finite sets $(V(G), E(G))$ such that $E(G)$ is a
subset of the set of unordered pairs of $V(G)$. The set $V(G)$ is the set of \emph{vertices}
and the set $E(G)$ is called the set of \emph{edges}.

If $e = \{u, v\}$ is an edge of $G$ one says that the vertices $u$ and $v$ are \emph{adjacent}. 

A \emph{walk} of length $n$ in $G$ is an alternating sequence of vertices and edges,
written as $w = \{v_0, z_1, v_1, \ldots, v_{n-1}, z_n, v_n\}$, where $z_i = \{v_{i-1}, v_i\}$ is the edge joining the vertices $v_{i-1}$ and $v_i$. A walk may also be written $\{v_0,\ldots, v_n\}$
with the edges understood, or $\{z_1, z_2, \ldots, z_n\}$ with the vertices understood.
If $v_0 = v_n$, the walk $w$ is called a \textit{closed walk}. A \textit{path} is a walk with all its
vertices distinct. A \textit{cycle} of length $n$ is a closed path $\{v_0, \ldots, v_n\}$ in which $n\ge  3$.
A \emph{forest} is an acyclic graph.\\

To each simple graph $G$ on the vertex set $V(G)$ we can associate a squarefree ideal $I(G)$ of the polynomial ring $S=K[x_1,\dots,x_n]$, called the \textit{edge ideal} associated to $G$,  defined as follows \cite{RV}:
$$
I(G) = (x_ix_j : \textrm{$x_i$ is adjacent to $x_j$}) = (x_ix_j:\{i,j\}\in E(G)).
$$
 We quote the next definition from \cite{RV}.
\begin{Def}
 An \emph{induced matching} in a graph $G$ is a set of pairwise disjoint edges $e_1,\ldots, e_r$ such that the only edges of $G$ contained in $\bigcup_{i=1}^re_i$ are $e_1, \ldots, e_r$. The \emph{induced matching number}, denoted by $\im(G)$, is the
number of edges in the largest induced matching.
	\end{Def}
	
\begin{Prop}\label{prop:forest}\textup{\cite[Corollary 7.5.6.]{RV}} If $G$ is a forest, then
$$
\reg(I(G)) = \im(G) + 1.
$$
\end{Prop}

See \cite{RV} for detailed information on this subject.

\begin{Cor}\label{cor:forest} Let $n$ be a positive integer and let $G$ be a simple graph on $[n]$. Assume
\[E(G)=
\begin{cases}
 \big\{e_i=\big\{i,i+\tfrac{n}{2}\big\}:i=1,\dots,\tfrac{n}{2}\big\} & \text{if $n$ is even}, \\\\

\big\{e_i=\big\{i,i+\left\lfloor\tfrac{n}{2}\right\rfloor\big\}:i=1,\dots,\left\lfloor\tfrac{n}{2}\right\rfloor\big\}\cup\big\{\{1,n\}\big\} & \text{if $n$ is odd}.
\end{cases} \]
	Then
	$$\reg(I(G)) = \left\lfloor\tfrac{n}{2}\right\rfloor + 1.$$
\end{Cor}
\begin{proof}
Indeed, in both cases $G$ is a forest and the assertion follows from Proposition \ref{prop:forest}.
\end{proof}

Now, we are ready to state and prove the analogous of Theorem \ref{thm:maximumregtspread} for the regularity of $t$--spread ideals generated in degree two.\smallskip

\begin{Cor}\label{thm:boundRegTSpreadEdgeIdeals}
	Let $n,t\ge 1$ with $n\ge 2t$ and let $I$ be a $t$--spread ideal equigenerated in degree two of $S=K[x_1,\dots,x_n]$. Then
	$$
	\reg(I)\ \le\ \left\lfloor\tfrac{n}{2}\right\rfloor+1.
	$$
\end{Cor}
\begin{proof}
	We distinguish two cases: $n$ even and $n$ odd.\\\\
	\textsc{Case 1.} Let $n$ be even. Following Theorem \ref{thm:maximumregtspread}, it suffices to find $\left\lfloor\frac{n}{2}\right\rfloor=\tfrac{n}{2}$ $t$--spread monomials of degree two, whose supports are pairwise disjoint and such that their union is $[n]$. For this purpose, we can consider the graph $G$ on $[n]$ with
	$$
	E(G)\ =\ \Big\{e_i=\big\{i,i+\tfrac{n}{2}\big\}:i=1,\dots,\tfrac{n}{2}\Big\}.
	$$
	The ideal $I(G)$ is the $t$--spread ideal minimally generated by the required $\left\lfloor\frac{n}{2}\right\rfloor=\tfrac{n}{2}$ $t$--spread monomials. From Corollary \ref{cor:forest} (see also Theorem \ref{thm:teorpart}) and Theorem \ref{thm:maximumregtspread}, we have 
	 
	$$\reg (I) \le \reg(I(G)) = \left\lfloor\tfrac{n}{2}\right\rfloor + 1 = n-(\tfrac{n}{2}-1).$$
	\textsc{Case 2.} Let $n$ be odd. 
	Let us consider the graph $G$ on the vertex set $[n]$, with
	$$
	E(G)\ =\ \Big\{e_i=\big\{i,i+\left\lfloor\tfrac{n}{2}\right\rfloor\big\}:i=1,\dots,\left\lfloor\tfrac{n}{2}\right\rfloor\Big\}\cup\big\{\{1,n\}\big\}.
	$$
From Corollary \ref{cor:forest}, we have that
\[\reg(I(G)) = \left\lfloor\tfrac{n}{2}\right\rfloor + 1.\]
On the other hand, since the structure of
\[I(G) = (x_ix_{i+\left\lfloor\tfrac{n}{2}\right\rfloor}:i=1,\dots,\left\lfloor\tfrac{n}{2}\right\rfloor)+(x_1x_n),\]
from Theorem \ref{thm:maximumregtspread}, it follows that
\[\reg (I) \le \reg(I(G)) = \left\lfloor\tfrac{n}{2}\right\rfloor + 1.\]
\end{proof}
\begin{Cor}
	Let $n\ge 2$ and $I$ an edge ideal of $S=K[x_1,\dots,x_n]$. Then
	$$
	\reg(I)\ \le\ \left\lfloor\tfrac{n}{2}\right\rfloor+1.
	$$
\end{Cor}

We close the section with a slight modification of the proof of Corollary \ref{thm:boundRegTSpreadEdgeIdeals} that, given three positive numbers $n,d,t\ge 1$ such that $n\ge 1+(d-1)t$, allows us to  obtain an optimal upper bound for the regularity of a $t$--spread ideal generated in degrees at most $d\ge1$.

\begin{Thm}\label{thm:boundRegTSpreadGeneratedInDegreeDAtMostD}
	Let $n,d,t\ge 1$ with $n\ge 1+(d-1)t$ and let $I$ be a $t$--spread ideal of $S=K[x_1,\dots,x_n]$ generated in degrees at most $d$. Then
	\begin{equation}\label{eq:generalizedBoundReg}
	\reg(I)\ \le\ n+1-\max\{\left\lceil\tfrac{n}{d}\right\rceil,t\}.
	\end{equation}
\end{Thm}
\begin{proof}
	Following Theorem \ref{thm:maximumregtspread}, it suffices to find the minimum possible number of $t$--spread monomials of degrees at most $d$ whose supports are pairwise disjoint and such that the union of their supports is $[n]$. We prove that this number is $\max\{\left\lceil\tfrac{n}{d}\right\rceil,t\}$.
We distinguish two cases.\\\\\
\textsc{Case 1.} $\left\lceil\tfrac{n}{d}\right\rceil\ge t$.
Suppose 
$n=\ell+d\left\lceil\tfrac{n}{d}\right\rceil$, $0\le\ell\le d-1$.
Using analogous techniques as in Proposition~\ref{prop:pasc}, we consider the following $t$--spread monomials
$$
u_i\ =\ \begin{cases}
x_ix_{i+\left\lceil\tfrac{n}{d}\right\rceil}\cdots x_{i+(d-2)\left\lceil\tfrac{n}{d}\right\rceil}x_{i+(d-1)\left\lceil\tfrac{n}{d}\right\rceil}&\text{if}\ i=1,\dots,\ell,\\
x_ix_{i+\left\lceil\tfrac{n}{d}\right\rceil}\cdots x_{i+(d-2)\left\lceil\tfrac{n}{d}\right\rceil}&\text{if}\ i=\ell+1,\dots,\left\lceil\tfrac{n}{d}\right\rceil.\\
\end{cases}
$$

Since $\left\lceil\tfrac{n}{d}\right\rceil\ge t$, the monomials $u_i$ are $t$--spread monomials of degrees at most $d$. 
Let us consider the ideal $J=(u_1, \ldots, u_{\left\lceil\tfrac{n}{d}\right\rceil})$. Since $J$ is the Pascal ideal of type $(n, \left\lceil\tfrac{n}{d}\right\rceil)$, 
then by Theorem \ref{thm:teorpart},
$$
\reg(J)\ =\ n-(\left\lceil\tfrac{n}{d}\right\rceil-1)\ =\ n+1-\left\lceil\tfrac{n}{d}\right\rceil.$$
Therefore, $J$ is the $t$--spread ideal generated in degrees at most $d$ with the maximum possible regularity. Indeed, the minimum possible number of generators of a squarefree ideal generated in degrees at most $d$, is $\left\lceil\tfrac{n}{d}\right\rceil$. \\\\
	\textsc{Case 2.} $\left\lceil\tfrac{n}{d}\right\rceil<t$. In such case, $n<dt$, thus $1+(d-1)t\le n<dt$.
	\\
	Hence, 
	$d-1\le\tfrac{n-1}{t}< d-\tfrac{1}{t}$, and so
	$$
	d-1\le\left\lfloor\tfrac{n-1}{t}\right\rfloor<d.
	$$
	Finally, $d\le\left\lfloor\tfrac{n-1}{t}\right\rfloor+1<d+1$. Therefore, $\left\lfloor\tfrac{n-1}{t}\right\rfloor+1=d$. By (\ref{eq:max}), $\ell=\left\lfloor\tfrac{n-1}{t}\right\rfloor+1=d$ is the maximum degree of a $t$--spread monomial of $S=K[x_1,\dots,x_n]$, and so the Pascal ideal of type $(n,t)$ gives the maximum regularity of a $t$--spread ideal of $S$ generated in degree at most $\ell=d$. Since $\left\lceil\tfrac{n}{d}\right\rceil<t$, we have $\max\{\left\lceil\tfrac{n}{d}\right\rceil,t\}=t$. Hence, $n+1-\max\{\left\lceil\tfrac{n}{d}\right\rceil,t\}=n-(t-1)$ and the bound given in (\ref{eq:generalizedBoundReg}) is true. The statement is proved.
\end{proof}
\begin{Rem}\rm
Theorem \ref{thm:boundRegTSpreadGeneratedInDegreeDAtMostD} generalizes Theorem \ref{thm:maximumregtspread}. Let $d$ be the maximum degree of a $t$--spread monomial of $S=K[x_1,\dots,x_n]$, then $\tfrac{n}{d}\le t$. Indeed, if $\tfrac{n}{d}>t$, then $n>dt$. Hence, $n\ge1+dt$ and $x_1x_{1+t}\cdots x_{1+dt}$ will be a $t$--spread monomial of $S=K[x_1,\dots,x_n]$ of degree $d+1>d$. An absurd by the meaning of $d$. It follows that $\tfrac{n}{d}\le t$. Thus $\left\lceil\tfrac{n}{d}\right\rceil\le t$ and $\max\{\left\lceil\tfrac{n}{d}\right\rceil,t\}=t$. Therefore, by equation (\ref{eq:generalizedBoundReg}), we get the maximum possible regularity of a $t$--spread ideal of $S$, \emph{i.e.}, $n+1-t$, and we obtain again Theorem \ref{thm:maximumregtspread}.
\end{Rem}


\begin{thebibliography}{99}

\bibitem{ATSpread21} L. Amata, Computational methods for $t$-spread monomial ideals, available at arXiv preprint \url{https://arxiv.org/abs/2110.11801} [math.AC].

\bibitem{AC2} L.~Amata, M.~Crupi, Extremal Betti Numbers of t--Spread Strongly Stable Ideals, \emph{Mathematics}. {\bf 7} (2019), 695.

\bibitem{AFC1} L.~Amata, A.~Ficarra, M.~Crupi, Upper bounds for Extremal Betti Numbers of $t$--Spread Strongly Stable Ideals, \emph{Bull. Math. Soc. Sci. Math. Roumanie} (to appear), available at  arXiv preprint \url{http://arxiv.org/abs/2102.07462}[math.AC].

\bibitem{AFC2} L.~Amata, A.~Ficarra, M.~Crupi, A numerical characterization of the extremal Betti numbers of t-spread strongly stable ideals, \emph{J. Algebr. Comb.}(2021). https://doi.org/10.1007/s10801-021-01076-0

\bibitem{AEL} C.~Andrei, V.~Ene, B.~Lajmiri, Powers of t--spread principal Borel ideals, \emph{Arch. Math.} {\bf 112}(6) (2019), 587--597. 

\bibitem{CAC} C.~Andrei--Ciobanu, Kruskal--Katona Theorem for $t$--spread strongly stable ideals, \emph{Bull. Math. Soc. Sci. Math. Roumanie} {\bf 62}(110)(2) (2019), 107--122.

\bibitem{AHH2} A.~Aramova, J.~Herzog, T.~Hibi, Squarefree lexsegment ideals,  \emph{Math.Z.} {\bf 228} (1998), 353--378.

\bibitem{BCP} D.~Bayer, H.~Charalambous, S.~Popescu, Extremal Betti numbers and Applications to Monomial Ideals, \emph{J. Algebra} {\bf 221} (1999), 497--512.

\bibitem{CLT} L.~Chu, S.~Liu, Z.~Tang, Castelnuovo--Mumford regularity and projective dimension of a squarefree monomial ideal, \emph{Front. Math. China} (2017), 1--10
\bibitem{RD} R.~Dinu, \emph{Gorenstein $T$--spread Veronese algebras}, Osaka J. Math. {\bf 57}(4) (2020), 935--947.

\bibitem{DHQ} R.~Dinu, J.~Herzog, A. A.~Qureshi, Restricted classes of veronese type ideals and algebras. \emph{Internat. J. Algebra Comput}. {\bf 31}(01) (2021), 173--197.

\bibitem{EHGB} V. Ene, J. Herzog, \emph{Gröbner bases in Commutative Algebra, Graduate Studies in Mathematics} {\bf 130} (American Mathematical Society, Providence, RI, 2011).

\bibitem{EHQ} V.~Ene, J.~Herzog, A.~A.~Qureshi, t-spread strongly stable monomial ideals, \emph{Comm. Algebra} {\bf 47(12)} (2019), 5303--5316.
			
\bibitem{GDS} D.~R.~Grayson, M.~E.~Stillman, \emph{Macaulay2, a software system for research in algebraic geometry}, available at \url{http://www.math.uiuc.edu/Macaulay2}.

\bibitem{JT} J.~Herzog, T.~Hibi, \emph{Monomial ideals} {\bf 260}(Graduate texts in Mathematics, Springer--Verlag 2011).

\bibitem{HHO} J.~Herzog, T.~Hibi, H.~Ohsugi, \emph{Binomial Ideals} {\bf 279} (Graduate Texts in Mathematics book series, Springer--Verlag 2018).

\bibitem{Ei} D.~Eisenbud, \emph{Commutative Algebra with a view toward  Algebraic Geometry} {\bf 150} (Graduate Texts in Mathematics book series, 1995).

\bibitem{HPV} J.~ Herzog, D.~Popescu, M.~ Vladoiu, Stanley depth and size of a monomial ideal, \emph{Proc. Am. Math. Soc.} {\bf 140}(2) (2011), 493--504.

\bibitem{HOC} M. Hochster, \emph{Cohen--Macaulay rings, combinatorics, and simplicial complexes}, Ring theory, II (Proc. Second Conf., Univ. Oklahoma, Norman, Okla., 1975), Lecture Notes Pure Appl. Math. {\bf 26} (1977).

\bibitem{GL} G.~ Lyubeznik, On the arithmetical rank of monomial ideals, \emph{J. Algebra} {\bf 112} (1998), 86--89.
			
\bibitem{MS}  E.~Miller, B.~Sturmfels. \emph{Combinatorial Commutative Algebra} {\bf 227} (Graduate Texts in Mathematics, Springer--Verlag, 2005).

\bibitem{JJR} J. J. Rotman, \emph{A First Course in Abstract Algebra}, 3rd ed. (Prentice Hall, Upper Saddle River, NJ, 2006).
			
\bibitem{TER} N. Terai, Alexander duality theorem and Stanley-Reisner rings in \emph{Free resolutions of coordinate rings of projective varieties and related topics} (Japanese) (Kyoto, 1998). S\"urikaisekikenky\"usho K\"oky\"uroku no. 1078 (1999), 174--184.

\bibitem{RV} R. H. Villareal. \emph{Monomial Algebras}, 2nd Edition (Chapman and Hall/CRC, 2018).

\bibitem{YP} A. A. Yazdan Pour. Resolutions and Castelnuovo--Mumford regularity, Differential Geometry [math.DG]. Universit\'e de Grenoble; Institute for advanced studies in basics sciences, 2012.

\end{thebibliography}
\end{document}